\documentclass[12pt,titlepage]{article}
\usepackage{amsmath}
\usepackage{amssymb}
\usepackage{epic,eepic}
\usepackage{verbatim}
\usepackage{graphicx}

\def\R{\mathbb{R}}

\def\al{\alpha}

\def\lm{\lambda}
\def\norm#1{\lVert#1\rVert}
\def\abs#1{\lvert#1\rvert}
 
\def\eps{\varepsilon}

\def\gm{\gamma}
\def\TGS{T_{GS}}

\def\wt{\widetilde}

\usepackage{amsthm}
\newtheorem{lem}{Lemma}[section]
\newtheorem{thm}{Theorem}[section]

\newtheorem{exl}{Example}

\theoremstyle{remark}
\newtheorem{rem}{Remark}[section]

\renewcommand{\baselinestretch}{2.0}

\begin{document}
\title{Acceleration Operators in the Value Iteration Algorithms for Average Reward Markov Decision Processes \\} \vskip5mm
\author{Oleksandr Shlakhter, Chi-Guhn Lee \\ \\
Department of Mechanical and Industrial Engineering \\
University of Toronto, Toronto, Ontario, Canada,  M5S 3G8}
\maketitle
\begin{abstract}

One of the most widely used methods for solving average cost MDP problems
is the  value iteration method. This method, however, is often computationally
impractical and restricted in size of solvable MDP problems. We propose
acceleration operators that improve the performance of the value iteration
for average reward MDP models. These operators are based on two important
properties of Markovian operator: contraction mapping and monotonicity. It
is well known that the classical relative value iteration methods for
average cost criteria MDP do not involve the max-norm contraction or
monotonicity property. To overcome this difficulty we propose to combine
acceleration operators with variants of value iteration for stochastic
shortest path problems associated average reward problems.


\end{abstract}

\section{Introduction} \label{sec:intro}
 One of the most widely used methods for solving average MDP problems is the value iteration method. In general, this method is often appears to be computationally impractical and restricted in size of solvable MDP problems. Shlakhter, et al. \cite{shlakhter:2005} proposed acceleration operators to speed up the convergence of value iteration algorithms for discounted MDP models. These operators were based on the contraction property of Markovian operators for discounted MDPs. In this paper we will show how similar techniques can be used to accelerate the convergence of the value iteration for average reward MDP models.

It is well known that the
classical relative value iteration methods for average cost criteria
MDP, unlike the discounted and the expected total-cost (stochastic shortest path)
models do not involve the max-norm contraction or monotonicity
property. Bertsekas  ~\cite{bertsekas:2001,
bertsekas:1998} proposed an elegant way of constructing variants of
the relative value iteration algorithm using the connection between the
average cost models and corresponding stochastic shortest path
problems, which possesses the above properties under assumption that all
policies are unichain and there exists a recurrent state under all
policies. We will show how this approach, combined with acceleration operators technique, can be used to get more efficient variants of value iteration for average reward MDP models.

The rest of the paper is organized as follows. In Sections \ref{sec:accel-VI-disc} and  \ref{sec:bert-VI} we  briefly present the accelerated operators method for discounted MDPs and Bertecas' variants of value iteration algorithm for the average reward MDP. In Section  \ref{sec:accel-VI-av} we introduce variants of accelerated value iteration algorithms for average reward MDPs. Section 3 presents the numerical studies of the proposed accelerated value iteration algorithms. In Section 4 we discuss the computational complexity of accelerated operators. Finally, Section 5 concludes the paper.

\section{Accelerated Value Iterations Algorithms}
\label{sec:accel-value}

\subsection{Accelerating Operators (Infinite Horizon Discrete Time Discounted MDP Case)}
\label{sec:accel-VI-disc}
Consider an infinite horizon Markov decision process (MDP) with a finite set of
states denoted by $S$, a finite set of actions $A(i)$ for each state $i \in S$, an
immediate reward $r(i,a)$ for each $i \in S$ and $a \in A=\cup_{i \in S}A(i)$, and
transition probabilities $p_{ij}(a)$ for the $i,j \in S$ and $a\in A(i)$. The
objective is to determine $v_i$, the minimum expected total discounted reward over an
infinite horizon starting in state $i$, where $\alpha$ is the discount factor ($0\le
\alpha < 1$). It is well known~\cite{Puterman:MDP94} that $v$ satisfies the
optimality equation
\begin{equation} \label{MDP}
v(i)=\min_{a\in A(i)}\left\{r(i,a)+\alpha \sum_{j\in S} p_{ij}(a) v(j)\right\}.
\end{equation}
The optimality equation given in Equation
(\ref{MDP}) can be written, with the definition of the operator $T$ on $U$, in the
following vector notation.
\begin{equation} \label{MDP vector}
v=Tv\equiv \min_{d\in \Pi}\left\{r_d + \alpha P_d v \right\},
\end{equation}
where $\Pi$ is the set of policies.
There are several standard methods for finding optimal or approximately optimal
policies for the discounted MDP models. Approaches widely employed to solve MDP problems include value
iteration, policy iteration, and linear programming approach~\cite{Puterman:MDP94}.
 Shlakhter et al. \cite{shlakhter:2005} proposed accelerating operators to improve the convergence of the value iteration algorithm. Let us briefly discuss this technique. Consider the linear programming formulations for discounted MDPs
\begin{equation} \label{eq:LP}
\max \left\{\sum_i h(i) \mid{}
h(i)-\alpha \sum_{j=1}^{n}p_{ij}(a)h(j)\le{}r(i,a), \forall i\in
\overline{1,n}, \forall a\in A(i), h \in \R^{n}. \right\}
\end{equation}

Let $V$ be the feasible set of linear program. It can be determined by thesystem of inequalities
$$V= \left \{h \mid{}
h(i)-\alpha \sum_{j=1}^{n}p_{ij}(a)h(j)\le{}r(i,a), \forall i\in
\overline{1,n}, \forall a\in A(i), h \in \R^{n} \right\}.$$ Let
$T$ be the Markovian operator
$$
    T h(i)=\min_{a \in A(i)}\left \{ r(i,a) + \alpha \sum_{j=1}^{n}p_{ij}(a)h(j) \right\}
$$

 The most crucial observation, that leads to
characterization of the acceleration operators, is that the set
$V$ is invariant under $T$, as formally stated in
Lemma \ref{lem:invariance_disc}.

\begin{lem}
  \label{lem:invariance_disc}
  $V$ is invariant under $T$. That is, $T V \subset V$.
\end{lem}

Lemma \ref{lem:invariance_disc}
suggests the following conditions for acceleration operator $Z$ on
defined set $V$.

\bigskip

\noindent {\bf Acceleration Conditions}
\begin{description}
\item [(A)] $Z V\subset V$,
\item [(B)] $Z v \ge v, \; \forall v \in V$.
\end{description}
\bigskip

\begin{rem}
 \label{r:convergence_disc}
It is easy to show that if $v^*$ is the fixed point of operator $T$, $v^{n+1}=
Tv^{n}$, $w^{n+1}=ZTw^{n}$, and $w^0=v^0$ then
  \begin{enumerate} \vspace{-0.2in}
  \renewcommand{\labelenumi}{(\roman{enumi})}
  \item $v^*\ge{}w^n\ge{}v^n$,\vspace{-0.2in}
  \item $v^*=\lim_{n\to\infty}w^n=\lim_{n\to\infty}v^n$,\vspace{-0.2in}
  \item $\norm{v^*-w^n}\le\norm{v^*-v^n}$.\vspace{-0.2in}
  \item The sequence $\{w^n\}$ converges globally with order 1 at a rate less than or equal to $\alpha$; its global asymptotic average rate of convergence is less than or equal to $\alpha$.
  \end{enumerate}
\end{rem}

In \cite{shlakhter:2005} we presented two particular realization of acceleration
operator $Z$: Projective Operator and Linear Extension Operator.

For a given operator satisfying two conditions {\bf (A)} and {\bf
(B)}, several variants of accelerated value iteration algorithm were
be suggested. More detailed discussion regarding application of this method to discounted MDP models can be found in the original paper. Further we will present the extended explanation of this technique applied to average reward MDPs.

\subsection{Bertsecas Approach}
\label{sec:bert-VI}

The most important properties of the Markovian operator $T$ of discounted or total cost(stochastic shortest path) used in the previous section are contraction property with respect to max-norm and monotonicity property. It is well known, however,  that the classical value iteration methods for average cost criteria MDP, unlike discounted and total cost models, do not involve the max-norm contraction or monotonicity property.  This means that direct application of acceleration operators technique to average MDP models is impossible.
Bertsekas \cite{bertsekas:2001, bertsekas:1998}  proposed a way of constructing the variants of
the relative value iteration algorithm, using the connection between the
average cost models and the corresponding stochastic shortest path
problem. Under assumption that all policies are unichain and there exists a recurrent state under all policies, the Markovian operator is a contraction mapping with respectto a weighted $\max$-norm, and possesses the monotonicity property. We will show how this approach combined with acceleration operators technique can be used to get more efficient variants of value iteration for average reward MDP models.

Let us briefly describe the approach proposed by Bertsakas \cite{bertsekas:2001,
bertsekas:1998}. Consider an infinite horizon Markov decision process (MDP) with a finite set of states denoted by $S$, a finite set of actions $A(i)$ for each state $i \in S$, an
immediate reward $r(i,a)$ for each $i \in S$ and $a \in A=\cup_{i \in S}A(i)$, and a
transition probability $p_{ij}(a)$ for each $i,j \in S$ and $a\in A(i)$. The
objective is to determine $\lm^*$, the minimum average cost per state over an
infinite horizon starting in state $i$, which satisfies the
optimality equation
\begin{equation} \label{MDP}
\lm^* +{h(i)}^*=\min_{a\in A(i)}\left\{r(i,a)+ \sum_{j\in S} p_{ij}(a) {h(j)}^*\right\},
\end{equation}
where $h^*$ is a differential vector.

Let us assume that there is a state, denoted by $n$, which is a recurrent state
under any stationary policy. Consider the stochastic shortest path problem (denoted $\lm$-SSP) obtained
from original model by adding a new state $t$ such that ${p }'_{it}=p_{in}$ for each $i \in S$, ${p }'_{ij}=p_{ij}$ for each $i \in S$, $j \in S\setminus \{n\}$, and ${p }'_{in}=0$ for each $i \in S$, where ${p }'_{ij}$ is the transition probability of the new model. The costs of the new model will
be equal to $r_i(a) -\lm$ for each $i \in S$, where $\lm$
is a scalar parameter, and $r_t(\cdot)=0$ for terminating state $t$.
Let $h_{\pi, \lm}(i)$ be the total expected cost of stationary
policy $\pi$ starting from state $i$, and let
the function $h_{\lm} (i)= \max_{\pi} h_{\pi,\lm}(i),~~~i=1,...,n$. One can
show that the functions $h_{\lm}(i)$ are concave, monotonically decreasing, and
piecewise linear as functions of $\lm$, and that
\begin{equation}\label{eq:4.56}
    h_{\lm} (n)= 0 ~~\text{if and only if} ~~~\lm=\lm^*
\end{equation}
Furthermore, the vector $h_{\lm^*}$, together with $\lm^*$,
satisfies the optimality equation
\begin{equation}\label{eq:optimality}
   \lm e+h_{\lm^*}=Th_{\lm^*}.
\end{equation}

As it was shown in ~\cite{bertsekas:2001, bertsekas:1998} $\lm ^*$
can be found using one-dimensional search procedure. It requires
solution of several associated stochastic shortest path problems, for which
value of parameter $\lm$ is updated as
\begin{equation}\label{eq:lm}
    \lm^{k+1}=\lm^k+\gm^k h_{\lm^k}(n)
\end{equation}
where $h_{\lm^k}(n)$ is the optimal solution of $\lm^k$-SSP which
can be found using the value iteration in form
\begin{equation}\label{eq:4.58}
    h^{m+1}(i)=\min_{a \in A(i)}\left \{ r(i,a)-\lm^k + \sum_{j=1}^{n-1}p_{ij}(a)h^m(j)
    \right\} \text{ for all } i\in \overline{1,n}.
\end{equation}
with $\lm^k$ fixed throughout the value iteration procedure.

\begin{rem} \label{r:step-size}
As it was shown in~\cite{bertsekas:2001, bertsekas:1998}, the sequence of iterates $(\lm^k,h_{\lm^k}(n))$
converges to $(\lm^*, h_{\lm^*}(n))$, provided that stepsize $\gm^k
\le \frac{1}{max_{\pi}N_{\pi}(n)}$, where $N_{\pi}(i)$ is the expected value of the first positive time that
$n$ is reached under $\pi$ starting from state $i$. It is also easy to see that, if additionally $h^0(n)\le 0$, then the sequence $\lm^k$ converges monotonously to $\lm^*$, which is $\lm^k \downarrow \lm^*$.
\end{rem}

The more efficient form of the algorithm proposed in
~\cite{bertsekas:2001, bertsekas:1998} is to update parameter
$\lm^k$ for each iteration
\begin{equation}\label{eq:4.59}
    h^{k+1}(i)=\min_{a \in A(i)}\left \{ r(i,a)-\lm^k + \sum_{j=1}^{n-1}p_{ij}(a)h^n(j) \right\} \text{ for all } i\in \overline{1,n}
\end{equation}
where the parameter $\lm^{k+1}=\lm^k + \gamma^k h^{k+1}(n)$,
$\gamma^n$ is a positive, sufficiently small step size, and
$h^{k+1}$ is a current approximation of optimal solution of
$h_{{\lm}^{k+1}}$ of corresponding $\lm^{k+1}$-SSP. It was also
proposed the improved variant of above algorithm, which is based on
the following inequality
\begin{equation}\label{eq:4.60a}
   \underline{\beta}^k \le \lm \le \overline{\beta}^k
\end{equation}
where
\begin{equation}\label{eq:4.60b}
   \underline{\beta}^k = \lm^k + \min \left[\min_{i\neq
   n}[h^{k+1}(i)-h^k(i)],h^{k+1}(n)\right],
\end{equation}

\begin{equation}\label{eq:4.60c}
\overline{\beta}^k = \lm^k + \max \left[\max_{i\neq
   n}[h^{k+1}(i)-h^k(i)],h^{k+1}(n)\right]
\end{equation}

Using this inequality it is possible to replace the iteration
$\lm^{k+1}=\lm^k + \gamma^k h^{k+1}(n)$ by $\lm^{k+1}=\Pi_k
\left[\lm^k + \gamma^k h^{k+1}(n) \right]$, where $\Pi_k \left[c
\right]$ denotes the projection of a scalar $c$ on the interval

\begin{equation}\label{eq:4.60d}
   \left[ \max_{m=0,...,k}\underline{\beta}^m, \min_{m=0,...,k} \overline{\beta}^m \right],
\end{equation}

\begin{rem} \label{r:term.-state}
Since the absorbing state $t$ has 0 cost, the cost of any policy
starting from $t$ is 0. Because of this we can ignore the the
component corresponding to the state $t$ and exclude it from
summation. Since for all states $i$  the transition probabilities
$p_{in}(a)=0$ the component corresponding to the state $n$ can be
also ignored.
\end{rem}

\begin{rem} \label{r:dimension}
Though the state space of the original stochastic shortest path problems
has the dimension $n+1$, taking into consideration the statement of
first part of remark ~\ref{r:term.-state} and because the Markovian
operator (\ref{eq:4.58}) and (\ref{eq:4.59}) does not change the value
in state $t$, we will consider this problem as problem in
$n$-dimensional state space.

\end{rem}
The proof of the convergence of this algorithm can be found in
original paper (see \cite{bertsekas:1998}).

Using this lemma it is easy to show that
\begin{equation}\label{eq: fixedpoint}
    \norm{h^k-h_{\lm^*}} \le
    \norm{h^k-h_{\lm^k}}+\norm{h_{\lm^k}-h_{\lm^*}} \le
    \norm{h^k-h_{\lm^k}}+O(\abs{\lm^k-\lm^*})
\end{equation}

It was shown in ~\cite{bertsekas:2001, bertsekas:1998} that both of sequences $h^k$ and $\lm^k$ converge
at rate of a geometric progression.

One of the disadvantages of the proposed algorithms is that the rate of
convergence of these methods is relatively slow. To bypass this we propose the acceleration
operators technique introduced in ~\cite{shlakhter:2005} to
accelerate convergence of discounted MDP models. As we will
show, this approach, which speeds up the convergence of value
iteration algorithms, can be very efficient for solving average cost
MDPs.

\subsection{Accelerating Operators Approach (Average Reward MDP Case)}
\label{sec:accel-VI-av}
In this section we show how the acceleration operators technique can be applied to average reward MDP models

\subsubsection{Reduction to Discounted Case}
\label{sec:reduction}
The acceleration operators can be directly applied to the following special case of the average MDP \cite{bertsekas:2001}. Assume that there is a state $t$ such that for some $\beta >0$ we have $p_{it}(a) \ge \beta$ for each $i \in S$. One can see that $(1-\beta)$-discounted problem with the same state space, and actions, and transition probabilities
\begin{equation}\label{eq:reduction}
    \overline{p}_{ij}(a)=\bigg \{\begin{array}{lc}
                                   (1- \beta)^{-1}p_{ij}(a) &\text {         if } j\neq t, \\
                                   (1-\beta)^{-1}(p_{ij}(a)-\beta) &\text { if } j= t.
                                 \end{array}
\end{equation}
Then $\beta \overline{v}(t)$ and $\overline{v}(i)$ are optimal average and differential costs, where $\overline{v}$ is the optimal cost function of the corresponding $(1-\beta)$-discounted problem.

Another straightforward application of the acceleration operators is to use the following relationship
\begin{equation}\label{eq:limit}
    \lm^* = \lim_{\alpha\rightarrow 1}(1-\al)v_{\al}(i)
\end{equation}
for each state $i \in S$, where $v_{\al}$ is an optimal cost vector for the corresponding $\al$-discounted problem.

Let's briefly discuss this relation. Getting a good approximation of the average cost $\lm^*$ using formula (\ref{eq:reduction}) requires calculation of the expected total discounted cost for discount factor close to 1. In general, for an MDP with the discount factor close to 1 this problem is considered to be computationally very demanding. Accelerating operators show a highly efficient way of solving such MDPs, making this relation useful for finding the solutions of the average cost MDP problems. This relation can also be used to obtain the upper and lower bounds for optimal average cost. It is well known \cite{Puterman:MDP94} that total discounted reward can be expressed in terms of optimal average cost and optimal bias $h$ as
 \begin{equation}\label{eq:bias}
    v_{\al}=(1-\al)^{-1}\lm+h+f(\al)
 \end{equation}
where  $h\equiv H_p r$, $H_p$ is a fundamental matrix $H_p \equiv (I-P-P^*)^{-1}(I-P^*)$, and $P^*= \lim_{N \rightarrow \infty}\frac{1}{N}E_s\{ \sum_{t=1}^N r(X_t) \}$, and  $f(\al)$ is a vector which converges to zero as $\al\uparrow1$. The additional property of the relationship (\ref{eq:bias}) can be formulated as following lemma

\begin{lem}
  \label{lem:bounds}
 For any MDP problem there exists $\eps$ ($0 \le \varepsilon < 1$) such that for any $\al \in (\eps, 1)$ there exist two coordinates $i$ and $j$ that $v_{\alpha}(i)< \lm^* < v_{\alpha}(j)$
\end{lem}
This means that $\min_{i}v_{\alpha}(i)$ and $\max_{i}v_{\alpha}(i)$ provide the lower and upper bound for $\lm^*$ for all $\al$ sufficiently close to 1. Though the above lemma doesn't provide guidance on how to choose the discount factor $\al$, the numerical studies show that taking the discount factor to be equal to 0.8 or higher gives us  sufficiently good results (see Table \ref{tab:2} of Section \ref{sec:numerical-studies}).

\subsubsection{Variants of Accelerated Value Iterations for Average Reward MDPs}
\label{sec:accel-VI-av-var}

In this section we propose several variants of accelerated value iteration algorithms.
Consider the following linear programming formulations, which are equivalent to
associated $\lm^k$-stochastic shortest path problems
\begin{equation} \label{eq:LP}
\max \left\{\sum_i h(i) \mid{}
h(i)-\sum_{j=1}^{n-1}p_{ij}(a)h(j)\le{}r(i,a)-\lm^k, \forall
i \in S, \forall a\in A(i), h \in \R^{n}. \right\}
\end{equation}

Let the set $H_{\lm^k}$ be determined by the system of inequalities shown below:
$$H_{\lm^k}= \left \{h \mid{}
h(i)-\sum_{j=1}^{n-1}p_{ij}(a)h(j)\le{}r(i,a)-\lm^k, \forall i \in S, \forall a\in A(i), h \in \R^{n} \right\},$$ and
$T_{\lm^k}$ be Markovian operator
$$
    T_{\lm^k}h(i)=\min_{a \in A(i)}\left \{ r(i,a)-\lm^k + \sum_{j=1}^{n-1}p_{ij}(a)h(j) \right\}
$$

  An observation, similar to Lemma \ref{lem:invariance_disc} of the previous section characterizing the acceleration operators, is that the set
$H_{\lm^k}$ is invariant under $T_{\lm^k}$ as formally stated in
Lemma \ref{lem:invariance}.

\begin{lem}
  \label{lem:invariance}
  $H_{\lm^k}$ is invariant under $T_{\lm^k}$. That is, $T_{\lm^k}H_{\lm^k} \subset H_{\lm^k}$.
\end{lem}

The following statement holds:

\begin{lem}
  \label{lem:invariance1}\
If $\lm^{k+1} \le \lm^k$, then  $H_{\lm^k} \subset H_{\lm^{k+1}}$.
If $\lm^{k+1} < \lm^k$, then  $H_{\lm^k} \subset
int(H_{\lm^{k+1}})$.
\end{lem}
 It is easy to notice that if $h^{k+1}(n) \le 0,$  then
$\lm^{k+1}=\lm^k+{\gm}^kh^{k+1}(n) \le \lm^k$. These two lemmas
suggest the following conditions for acceleration operator $Z_k$ on
defined set $H_{\lm^k}$.

\bigskip

\noindent {\bf Acceleration Conditions}
\begin{description}
\item [(A)] $Z_k H_{\lm^k}\subset H_{\lm^{k}}$,
\item [(B)] $Z_{k} h > h, \; \forall h \in H_{\lm^k}$.
\end{description}
\bigskip

The properties of operator $Z_k$ satisfying the conditions {\bf (A)}
and {\bf (B)} are given in the following lemma
\begin{lem}
 \label{l:convergence}
If $h^0 \in H_{\lm^0}$, $\lm^k \downarrow\lm^* $, $h^{k+1}=
T_{\lm^k}h^{k}$, $\wt{h}^k =Z_{k} h^k$, then
  \begin{description} \vspace{-0.2in}
  \renewcommand{\labelenumi}{(\roman{enumi})}
  \item [(i)] $h^*=\lim_{k\to\infty}h^k=\lim_{k\to\infty}\widetilde{h}^k$,\vspace{-0.2in}
  \item [(ii)] $h^*\ge h_{\lm^k}\ge{}\wt{h}^k \ge {}h^k$,\vspace{-0.2in}
  \item [(iii)] $\norm{h^*-h^k} \ge \norm{h^*-\wt{h}^k}$, where $\norm{\cdot}$ is a weighted $\max$-norm.\vspace{-0.2in}
  \end{description}
\end{lem}

Later on we will present two particular realizations of acceleration
operator $Z$.

For a given operator satisfying two conditions {\bf (A)} and {\bf
(B)}, several variants of accelerated value iteration algorithms can
be suggested. Notice that different acceleration operators may be
used in different iterations of the value iteration algorithm, in
which case $Z_k$ is an acceleration operator used in iteration $k$,
instead of $Z$.

Consider a variant of VI where
$\lm$ is kept fixed in all iterations until we obtain $h_{\lm^k}$,
the optimal solution of $\lm^k$-SSP, and after that we calculate
$\lm^{k+1}=\lm^k+\gm^kh_{\lm^k}$.  The
acceleration operators $Z_k$ such that $\wt{h}^{k+1}=Z_k T_{\lm^k}\wt{h}^{k}$ make this algorithm computationally attractive. It is
motivated by the fact that for this variant of VI we can guarantee
the monotonicity of the sequence $h_{\lm^k}(n)$ for stepsize satisfying
the inequality $\gm^k \le \frac{1}{max_{\pi}N_{\pi}(n)}$ of Remark \ref{r:step-size}, and as corollary the
monotonicity of the sequence $\lm^k$. This method can be improved using the following way of $\lm$ update. If $\lm^k$ and $\lm^{k+1}$ are two approximations of the value $\lm^*$ such that $\lm^* < \lm^{k+1} \le \lm^k$ and $h_{\lm^k}(n) \ge h_{\lm^{k+1}}(n) > 0$, then it is guaranteed that $\tilde{\lm} \ge \lm^*$, where $\tilde{\lm}=\lm^k-\frac{h_{\lm^k}(n)(\lm^{k+1}-\lm^k)}{h_{\lm^{k+1}}(n)-h_{\lm^{k}}(n)}$ is a point of intersection of a straight line through points with coordinates $(\lm^k, h_{\lm^k}(n))$ and $(\lm^{k+1}, h_{\lm^{k+1}}(n))$ with $\lm$ axis on $(\lm, h_{\lm})$ graph. Then we can take $\lm^{k+2}= \min \{\lm^{k+1}+\gamma h_{\lm^{k+1}}(n), \tilde{\lm} \}$.  This $\lm$ update is graphically illustrated in Figure~\ref{f:stepsize}.

\vspace{1in}
\begin{center}
Figure~\ref{f:stepsize} goes here.
\end{center}
\vspace{1in}

A formal description of this variant of VI is given below:

\renewcommand{\baselinestretch}{1.5}
\noindent{\bf General Accelerated Value Iteration Algorithm 1 (GAVI
1)} \vspace{-0.4in}
\begin{quotation}
\texttt{
\begin{description}
  \item [Step 0] Select $\lm^0$ such that $h_{\lm^0}(n) <0$, $\wt{h}^0=h^0 \in H_{\lm^0}$,  set $k=0$, and specify $\eps >0$.
  \item [Step 1] Compute $h_{\lm^{k}}(i)= \lim_{m\rightarrow\infty}Z_k T_{\lm^k}\wt{h}^{m}(i)$ for all $i\in{I}$.
  \item [Step 2] If $\norm{h_{\lm^{k}}(n)}>\eps$, go to {\bf Step 3}. Otherwise compute $\lm^{k+1}= \min\{ \lm^k+ \gm^k h_{\lm^{k}}(n), \tilde{\lm}\}$, where $\tilde{\lm}=\lm^k-\frac{h_{\lm^k}(n)(\lm^{k+1}-\lm^k)}{h_{\lm^{k+1}}(n)-h_{\lm^{k}}(n)}$, increase $k$ by 1
  and return to {\bf Step 1}.
  \item [Step 3] Return with the actions attaining the minimum in {\bf Step 1}.
\end{description}
}
\end{quotation}
\renewcommand{\baselinestretch}{2.0}

\begin{rem}
 \label{r:bounds1}
The choice of $\lm^0$ satisfying the conditions of Step 1 is always
available,  because $\lm_{\min} \le \lm^* \le \lm_{\max}$, where
$\lm_{\min}=\min_{i \in I, a \in A(i)}r(i,a)$ and
$\lm_{\max}=\max_{i \in I, a \in A(i)}r(i,a)$. It is easy to see
that $h_{\lm_{\max}}(i) \le 0$ and $h_{\lm_{\min}}(i) \ge 0$ for all
$i \in I$.
\end{rem}

\begin{rem}
 \label{r:bounds2}
The performance can be significantly improved if $\lm^0$ is taken as a upper bound for $\lm^*$ from the statement derived at the end of the section \ref{sec:reduction}.
\end{rem}

 An alternative variant of accelerated VI, though not very efficient, can be obtained from (\ref{eq:4.59}), where the parameter $\lm^{k}$ is updated at every
iteration. The direct application of acceleration technique
to a variant with parameter $\lm^{k+1}=\lm^k + \gamma^k h^{k+1}(n)$ is
impossible, because the monotonicity of the sequence $\lm^k$ can not be
guaranteed. We propose to apply the acceleration step only for iterates for which  $h^k(n)$ remains negative. Although this method is not guaranteed to be applicable for all iterates, it can still be more efficient than the standard value iteration.
A formal description of this variant of VI is given below:

\renewcommand{\baselinestretch}{1.5}
\noindent{\bf General Accelerated Value Iteration Algorithm 2 (GAVI
2)} \vspace{-0.4in}
\begin{quotation}
\texttt{
\begin{description}
  \item [Step 0] Select $\lm^0$ such that $h_{\lm^0}(n) <0$, $\wt{h}^0=h^0 \in H_{\lm^0}$,  set $k=0$, and specify $\eps >0$.
  \item [Step 1a] Compute $\wt{h}^{k+1}(i)=Z_k T_{\lm^k}\wt{h}^{k}(i)$ for all
  $i\in{I}$ until $h^k(n)$ remains negative
  \item [Step 1b] Compute $\wt{h}^{k+1}(i)=T_{\lm^k}\wt{h}^{k}(i)$ for all
  $i\in{I}$ otherwise
  \item [Step 2] If $\norm{\wt{h}^{k+1}-\wt{h}^{k}}>\eps$, go to {\bf Step 3}. Otherwise increase $k$ by 1
  and return to {\bf Step 1}.
  \item [Step 3] Return with the actions attaining the minimum in {\bf Step 1}.
\end{description}
}
\end{quotation}
\renewcommand{\baselinestretch}{2.0}

Now we will introduce one more variant of the accelerated value iteration
algorithm. As we mentioned earlier, the functions $h_{\lm} (i)$ are concave, monotonically decreasing,  piecewise linear function of $\lm$, and $ h_{\lm} (n)= 0$ if and only if $\lm=\lm^*$. This means that $\lm ^*$ can be found using the one-dimensional search
procedure. This procedure may be computationally intractable, because it
requires the solution of several associated expected total cost
problems. Using the acceleration technique improves this approach
significantly. A variant of such search procedure can
be obtained from GAVI 1, where the parameter $\lm^{k}$ is updated as
$\lm^{k}=({\lm^k}_{\min}+{\lm^k}_{\max})/2$, where  lower and upper
bounds for $\lm^*$ are obtained at k-th iteration. Here some
explanations are required. We can take ${\lm^0}_{\min}={\lm}_{\min}$
and ${\lm^0}_{\max}={\lm}_{\max}$ defined in Remark \ref{r:bounds1} or
derived at the end of the section \ref{sec:reduction}. For $\lm^{0}=({\lm^0}_{\min}+{\lm^0}_{\max})/2$ find
$h_{\lm^0}$. If $h_{\lm^0}(n) \le 0$, then
${\lm^1}_{\min}={\lm^0}_{\min}$ and ${\lm^1}_{\max}={\lm}^1$,
otherwise ${\lm^1}_{\min}=\lm^1$ and ${\lm^1}_{\max}={\lm^0}_{\max}$
etc. A formal description of this variant of VI is given below:

\renewcommand{\baselinestretch}{1.5}
\noindent{\bf General Accelerated Value Iteration Algorithm 3 (GAVI
3)} \vspace{-0.4in}
\begin{quotation}
\texttt{
\begin{description}
  \item [Step 0] Select ${\lm^0}_{\min}={\lm}_{\min}$
and ${\lm^0}_{\max}={\lm}_{\max}$, $\wt{h}^0=h^0 \in
H_{{\lm^0}_{\max}}$, set $k=0$,  and specify $\eps
>0$.
  \item [Step 1] Select $\lm^{k+1}=({\lm^k}_{\min}+{\lm^k}_{\max})/2$, .
  \item [Step 2] Compute $h_{\lm^{k}}(i)= \lim_{m\rightarrow\infty}Z_k T_{\lm^k}\wt{h}^{m}(i)$ for all $i\in{I}$.
  \item [Step 3] If $\norm{h_{\lm^{k}}(n)}<\eps$, go to {\bf Step 4}. Otherwise if $h_{\lm^k}(n) < 0$, then
${\lm^{k+1}}_{\min}={\lm^k}_{\min}$ and
${\lm^{k+1}}_{\max}={\lm}^k$, otherwise ${\lm^{k+1}}_{\min}=\lm^k$
and ${\lm^{k+1}}_{\max}={\lm^k}_{\max}$,  increase $k$ by 1
  and return to {\bf Step 1}.
  \item [Step 4] Return with the actions attaining the minimum in {\bf Step 1}.
\end{description}
}
\end{quotation}
\renewcommand{\baselinestretch}{2.0}

Now we will propose an acceleration operator satisfying two
conditions {\bf (A)} and {\bf (B)}  with significant reduction in the number of iterations before convergence and little additional
computation in each iteration, so that the overall performance is greatly improved. Now we propose an acceleration operator that requires little additional computation per iteration but reduces the number
of iterations significantly.

\bigskip

\noindent {\bf Projective Operator}

\noindent For $h^k \in V_{\lm^k}$, $Z_k:h^k \rightarrow h^k+\al^* e$,
where $\al^*$ is the optimal solution of the following trivial 1-dimensional optimization problem:
\begin{equation} \label{eq:1DLP}
\max \left\{\sum h(i) + n \alpha   \mid T_{\lm^{k+1}} (h^k+\alpha e) \ge
h^k + \alpha e\right\}.
\end{equation}

\vspace{1in}
\begin{center}
Figure~\ref{f:PAVI} goes here.
\end{center}
\vspace{1in}

Having a single decision variable in the above optimization problem, it is straightforward to
find the optimal solution. The role of Projective Operator is
graphically illustrated in Figure~\ref{f:PAVI}, where $Z_k$ projects
the given point $h_k \in H_k$ to the boundary of $H_{k+1}$.
\begin{thm} \label{l:Pu_satisfy12}
For any index $k$ Projective Operator $Z_k$ satisfies the
conditions {\bf (A)} and {\bf (B)}.
\end{thm}

We will call GAVI 1, GAVI 2, and GAVI 3 with Projective Operator as Projective Accelerated
Value Iteration 1, 2, and 3, or PAVI 1, PAVI 2,and PAVI 3 for short in the paper.

Now we present another acceleration operator that satisfies {\bf
Acceleration Conditions} {\bf (A)} and {\bf (B)}.

\bigskip

\noindent {\bf Linear Extension Operator}

\noindent For $h^k\in H$, $Z_k:v\rightarrow h^k+\al^*(h^k-h^{k-1})$,
where $h^k=T_{\lm^k}h^{k-1}$ and $\al^*$ is the optimal solution to
the following linear program:
\begin{equation} \label{eq:1DLP2}
\max\left\{\sum h^k_i + \alpha\sum (h^k_i-h^{k-1}_i) \mid
T(h^k+\alpha (h^k-h^{k-1})) \ge h^k+\alpha (h^k-h^{k-1})\right\}.
\end{equation}

Figure \ref{f:LAVI} graphically illustrates how Linear Extension
Operator works. It casts $T_{\lm^k}h^k$ in the direction of
$Th^k-h^k$ to the boundary of the set $H_{\lm^{k+1}}$. Since $h^k \in
H_{\lm^k}$, we have $T_{\lm^k}h^k \ge h^k$, which is an improving
direction. As a result, Linear Extension Operator moves $Th^k$
closer to the point $h^*$.

\vspace{1in}
\begin{center}
Figure~\ref{f:LAVI} goes here.
\end{center}
\vspace{1in}

\begin{thm}\label{thm:L_w_satisfy12}
Linear Extension Operator $G$ satisfies the conditions {\bf (A)} and {\bf (B)}.
\end{thm}

When Linear Extension Operator $G$ is used in place of $Z$ in {\bf Step 1} of GAVI, we
call the algorithm Linear Extension Accelerated Value Iteration or LAVI for short in
the paper.

 Interesting aspect of the proposed approach is that it can be used in Gauss-Seidel variant of value iteration algorithms.

\textbf {Gauss-Seidel:} $h^{k+1}=T_{GS_{\lm^k}}{h^k}$ where
\begin{equation}
  h^{k+1}(i)=\min_{a\in{A(i)}}\left\{r(i,a)-\lm^k+\sum_{j<i} p_{ij}(a)
  h^{k+1}(j)+ \sum_{j \geq i}^{n-1} p_{ij}(a) h^k(j) \right\}, \; \forall i\in \overline{1,n}. \label{eq:TGS}
\end{equation}

We start with the following definition of set:
\begin{align*}
H_{GS_{\lm^k}}=\{h\in\R^{n}\mid h\le{}T_{GS_{\lm^k}}{h}\}.
\end{align*}

The following lemma is an analogue of Lemma~\ref{lem:invariance}:
\begin{lem}
  \label{lem:invariance of splittings}\
  $H_{GS_{\lm^k}}$ is invariant under $T_{GS_{\lm^k}}$.
\end{lem}

With Lemma \ref{lem:invariance of splittings} acceleration operators
satisfying conditions {\bf (A)} and {\bf (B)} can be used in {\bf
Step 1} of GAVI with the variants $T_{GS_{\lm^k}}$ of the standard
operator $T_{\lm^k}$. However, it is not trivial to define
$H_{GS_{\lm^k}}$ with a set of linear inequalities and the
acceleration operators proposed in this research will not work.
To avoid the problem, we restrict the acceleration operators to a
strict subset of $H_{GS_{\lm^k}}$ .

\begin{lem} \label{l:relationships_between_V's}
The following relation holds
\begin{equation}
 H_{\lm^k}\subset{H_{GS_{\lm^k}}}.\label{eq:3}
\end{equation}
\end{lem}

\begin{rem}
Gauss-Seidel methods require special consideration, since in general
$H_{\lm^k}\ne{}H_{GS_{\lm^k}}$ ~\cite{shlakhter:2005}.
\end{rem}

\begin{lem} \label{l:TX(VX)subsetV}
\begin{equation}
    T_{GS_{\lm^k}}( H_{GS_{\lm^k}}) \subset H_{\lm^k}.
\end{equation}
\end{lem}

\begin{thm} \label{thm:H_invariance_under_TGS}
The set $H_{\lm^k}$ is invariant under $T_{GS_{\lm^k}}$. That is,
\[
  T_{GS_{\lm^k}}{H_{\lm^k}}\subset{}H_{\lm^k}.
\]
\end{thm}

Theorem \ref{thm:H_invariance_under_TGS} states that $H_{\lm^k}$ is
invariant under $T_{GS_{\lm^k}}$, which suggests that this
operator can replace $T_{\lm^k}$ in GAVI to give rise to new
accelerated value iteration algorithms. Therefore, we obtain several
accelerated versions of the value iteration algorithm, which are conveniently
written in the form XAYN, where `X' is either ``P'' for ``Projective'' or ``L'' for ``Linear
Extension'', `A' is for ``Accelerated'', `Y' is either ``VI'' for VI, or ``GS'' for GS, and 'N' is for one of ``1'', ``2'', and ``3'' for ``GAVI1'', ``GAVI2'', and ``GAVI3''. For example,
LAGS1 denotes Linear Extension Accelerated Gauss-Seidel value iteration method of type 1 (as for GAVI1) with $w^{n+1}=ZT_{GS}$
for {\bf Step 1}. Non-accelerated versions will be shortened to VI, and GS
without prefixes.

\section{Numerical Studies}
\label{sec:numerical-studies}

In this section we present numerical studies to demonstrate the computational
improvement that the proposed variants of two-phase accelerated value iteration algorithms achieve. The results are compared with Bertsecas' approach. We will consider four families of
randomly generated MDP problems. In all cases the number of actions in each state, the immediate
rewards for each state, and actions were generated using a uniform random number
generator. In all examples, we first fixed the number of non-zero entries in each row, so that the
density of non-zero entries in that row is equal to a given density level. We
randomly generated non-zero entries according to a uniform distribution over (0,1),
normalized these non-zero entries so that they add up to 1, and then placed them
randomly across the row.

\begin{exl} \label{ex:dense1}
Consider MDPs with 50 states and up to 50 actions per state. The transition
probability matrices were generated using a uniform random number generator and
non-zero elements were uniformly placed in the matrix. The density of non-zero
elements of the matrices varies from 30\% to 90\%. For two-phase variants of VI (columns of 3-6) the corresponding discounted MDP with discount factor $\alpha =0.99$ is solved.
\end{exl}

\begin{exl} \label{ex:dense2}
Consider MDPs with 100 states and up to 20 actions per state. The transition
probability matrices were generated using a uniform random number generator and
non-zero elements were uniformly placed in the matrix. The density of non-zero
elements of the matrices varies from 60\% to 90\%. For two-phase variants of VI (columns of 3-6) the corresponding discounted MDP with discount factor $\alpha =0.99$ is solved.
\end{exl}

\begin{exl} \label{ex:dense3}
Consider MDPs with 80 states and up to 40 actions per state. The transition
probability matrices were generated using a uniform random number generator and
non-zero elements were uniformly placed in the matrix. The density of non-zero
elements of the matrices varies from 40\% to 90\%. For two-phase variants of VI (columns of 3-6) the corresponding discounted MDP with discount factor $\alpha =0.99$ is solved.
\end{exl}

\begin{exl} \label{ex:dense4}
Consider MDPs with 200 states and up to 30 actions per state. The transition
probability matrices were generated using a uniform random number generator and
non-zero elements were uniformly placed in the matrix. The density of non-zero
elements of the matrices varies from 50\% to 90\%. For two-phase variants of VI (columns of 3-6) the corresponding discounted MDP with discount factor $\alpha =0.99$ is solved.
\end{exl}

As it was shown in \cite{shlakhter:2005} that the combinations of Projective operators with standard value iteration and Linear Extension operators Gauss-Seidel variant of value iteration give the best performance. The computational results of Examples 1-4 are presented in Table~\ref{tab:1} and Table~\ref{tab:2}.

\vspace{1in}
\begin{center}
Table~\ref{tab:1} goes here.
\end{center}
\vspace{1in}

Let us now present the brief analysis of the above numerical results.
Based on Examples 1 - 4, we can conclude that the proposed variants of accelerated value iteration algorithm  PAVI 1-PAVI 3 and LAGS 1-LAGS 3 show good performance and converge   up to 75 times faster than corresponding Bertsecas variants of value iteration algorithms.

For almost all of the cases in Table 1, PAVI 3  is the best algorithm. PAVI 1 gives almost the same good results, while PAVI2 performs relatively poorly. On the other hand, for most cases in Table 2, LAGS 2 is the best algorithm, while both LAGS 1 and LAGS 3 also show very strong performance. Summarizing these numerical results, we may conclude that both PAVI 3 and LAGS 3 show very good performance. Besides, application of these methods does not require choosing a stepsize. Similarly, both PAVI 1 and LAGS 1 show good performance, though additional computational efforts are necessarily to obtain a stepsize. We should also notice that the performance of both PAVI 2 and LAGS 2 may vary. This can be explained by the fact that it is not guaranteed that accelerating operators can be applied for all iterates. Because this fact, these algorithms are sensitive to the choice of a stepsize and  to the structure of transition probability matrix.

Let us also notice that the performance of two-phase algorithms depends on the choice of the discount factor of the corresponding discounted MDP problem which is solved during the first phase. We have the following tradeoff: having the discount closer to 1 leads to the increase in the number of iterations for the first phase. On the other hand, it make the bounds for the optimal average reward more tight, which leads to the reduction of the number of iterations of the second phase. The values of upper and lower bounds of the optimal average reward of Examples 1-4 are presented in Table~\ref{tab:2}

\vspace{1in}
\begin{center}
Table~\ref{tab:2} goes here.
\end{center}
\vspace{1in}

Based on Examples 1-4, we can conclude that the two-phase accelerated value iteration algorithms with second phase as the standard value iteration combined with PAVI 1, PAVI2, and PAVI 3 show better performance than the value iteration algorithm proposed by Bertsecas. For the best cases the accelerated value iteration algorithms converge up to 12 times faster than corresponding Bertsecas algorithm. We can also conclude that two-phase accelerated value iteration algorithms with second phase as the standard value iteration combined with PAVI 1 for most of the cases perform better than one combined with PAVI 1 and PAVI 3.

\section{Computational Complexity and Savings}
We now evaluate the number of the additional arithmetic operation required for application of the proposed accelerated operators. For the standard VI,
when the transition probability matrices are fully dense, each iteration will take
$C|S|^2$ (where $C$ is the average number of actions per state) multiplications and
divisions. With sparse transition probability matrices, this number can be estimated
as $NC|S|$ (where $N$ is the average number of nonzero entries per row of the
transition probability matrices).

The additional effort required in GAVI is due to the acceleration operator used in
{\bf Step 1} of GAVI. However, as it was discussed in \cite{shlakhter:2005} the acceleration step requires only $C|S|$ multiplication and division. This means that the iteration of GAVI requires $C|S||S+1|$ multiplication and division, which is
just slightly more than that of standard value iteration. Now we will show that for stochastic shortest path problems the following reduction of computational complexity is possible. With either Projective
Operator or Linear Extension Operator, a
trivial 1-dimensional LP should be solved per iteration of GAVI. Let us evaluate the
complexity of this step. Substitute the expression $h^k(i)+
\alpha$ into system of inequalities defining the set $H_{\lm^{k+1}}$. We
will have the following system of inequalities
$$
h^k(1)+\alpha-\sum_{j=1}^{n-1}p_{1j}(a)(h^k(j)+\alpha)\le{}r(1,a)-\lm^{k+1}
$$
$$
\vdots
$$
$$
h^k(n)+\alpha-\sum_{j=1}^{n-1}p_{nj}(a)(h^k(j)+\alpha)\le{}r(n,a)-\lm^{k+1}
$$

which can be written in the following form
$$
\alpha(1-\sum_{j=1}^{n-1}p_{1j}(a))
\le{}r(1,a)-\lm^{k+1}-(h^k(1)-\sum_{j=1}^{n-1}p_{1j}(a)h^k(j))
$$
$$
\vdots
$$
$$
\alpha(n-\sum_{j=1}^{n-1}p_{1j}(a))
\le{}r(n,a)-\lm^{k+1}-(h^k(n)-\sum_{j=1}^{n-1}p_{nj}(a)h^k(j))
$$

It is easy to notice that if $h^k \in H_k$ then the expressions in
righthand side of all inequalities are positive. It is also easy to
notice that for state $i$ and action $a \in A(i)$ such that
$p_{it}(a)=0$ we have lefthandside expressions equal to 0 and these
inequalities are satisfied for all value of $\alpha$. So, we have to
evaluate the ratios only for states and actions for which
$p_{it}(a)\neq 0$. For many real applications the number of such
states $K$  can be significantly less than $n$. So, the acceleration step requires only
$CK$ multiplications and divisions. Therefore, each iteration of GAVI may be just slightly more
expensive than the standard value iteration. In conclusion, the additional computation due to the accelerating operators is
marginal.

\section{Conclusions} \label{sec:conclusions}
Using the monotone behavior of the contraction mapping operator used in the
value iteration algorithm within the feasible set of the linear programming problem
equivalent to the discounted MDP and stochastic shortest path models, we propose a class of operators that can be used in combination with the standard contraction mapping and Gauss-Seidel methods to improve the computational efficiency. Two acceleration operators, Projective Operator and Linear Extension Operator, are particularly proposed and combined into the three variants of value iteration algorithms. The numerical studies show that the savings due to the acceleration have been essential and the maximum savings is up to 80 time faster than the case without our accelerating operator. It is especially interesting to mention that the savings become significant when the proposed variants of the value iteration algorithms for average reward MDPs are used with bounds obtained from solution of the corresponding discounted MDP problem.

\appendix

\section{Proofs} \label{sec:proofs}

\begin{lem}  \label{l:monotonicity}
\noindent \begin{enumerate}
  \renewcommand{\labelenumi}{(\roman{enumi})}
  \item  If $h\ge{}g$, then for any $\lm^k$ of GAVI
        $\;\; T_{\lm^k}h\ge{}T_{\lm^k}g$ and $T_{GS_{\lm^k}} h\ge{}T_{GS_{\lm^k}} g$.
  \item  If $h>{}g$, then for any $\lm^k$ of GAVI
  $\;\; T_{\lm^k}h>{}T_{\lm^k}g$ and $T_{GS_{\lm^k}} h>{}T_{GS_{\lm^k}} g$.
\end{enumerate}
\end{lem}
\begin{proof}[Proof of Lemma~\ref{l:monotonicity}]
Let us first prove that $g\ge{}h$ implies
$T_{\lm^k}g\ge{}T_{\lm^k}h$.
\[
T_{\lm^k}u(i)=\min_{a\in
A(i)}\left\{r(i,a)-\lm^k+\sum_{j}p_{ij}(a)g(j)\right\} \ge
\]
\[
\min_{a\in A(i)}\left\{r(i,a)-\lm^k+\sum_{j}p_{ij}(a)h(j)\right\} =
Th(i) \text{ for all } i\in \overline{1,N}.
\]
For part (ii) the inequality can be simply replaced with a strict
inequality. Let $f=\TGS g$ and $\xi=\TGS h$. Then $ f(1)=
(T_{GS_{\lm^k}} f)(1) = (T_{\lm^k} g)(1) \le (T_{\lm^k}h) (1) = \xi(1). $ By induction, assuming that $f(k) \le \xi(k)$
for all $k<i$, we get for
  $k=i$
  \begin{align*}
    f(i)&=\min_{a\in A(i)}\biggl(r(i,a)-\lm^k+\sum_{j<i}p_{ij}(a)f(j) + \sum_{j \ge i}p_{ij}(a)g(j) \biggr)\\
         &\le\min_{a\in A(i)}\biggl(r(i,a)-\lm^k+\sum_{j<i}p_{ij}(a)\xi(j) + \lm\sum_{j \ge i}p_{ij}(a)h(j) \biggr)=(T_{GS_{\lm^k}} h)(i)=\xi(i).
  \end{align*}

\end{proof}

\begin{proof}[Proof of Lemma~\ref{lem:bounds}]

Keeping in mind the additional property of optimal bias $P^* h = 0$, where $P^*$ is a positive matrix, and putting aside the trivial case $r(i,a)= 0$ for all $i$ and $a$, it is easy to see, that the optimal bias $h$ should have both positive and negative coordinates \cite{Lewis:2001}. Letting $\al$ approach to 1 in  (\ref{eq:bias}), we can make the term $f(\al)$ be arbitrarily small.
From this we can conclude that there exist two coordinates $i$ and $j$ that $v_{\alpha}(i)< \lm^* < v_{\alpha}(j)$ for all $\al$ sufficiently close to 1.
\end{proof}

\begin{proof}[Proof of Lemma~\ref{lem:invariance_disc}]
Proof of this lemma is identical to proof of lemma \ref{lem:invariance}, and can be found in the original paper \cite{shlakhter:2005}.
\end{proof}

\begin{proof}[Proof of Lemma~\ref{lem:invariance}]
Let $h \in H_{\lm^k}$ and $g=T_{\lm^k}h$. By definition of set
$H_{\lm^k}$, $g = T_{\lm^k}h \ge h$. By monotonicity shown in
Lemma~\ref{l:monotonicity}, $T_{\lm^k}g \ge T_{\lm^k}h=g$. Thus,
$g \in H_{\lm^k}$.
\end{proof}

\begin{proof}[Proof of Lemma~\ref{lem:invariance1}]
This lemma is a trivial corollary of a statement of Remark
\ref{r:monotonisity}. Let $h \in H_{\lm^k}$. By definition it
means that $$ h(i)-\sum_{j=1}^{N-1}p_{ij}(a)h(j)\le{}r(i,a)-\lm^{k},
\text{  for all  } i=\overline{1,N}.$$ From Remark
\ref{r:step-size} we have $\lm^k > \lm^{k+1}$, so $h$ satisfies
$$ h(i)-\sum_{j=1}^{N-1}p_{ij}(a)h(j) <{}r(i,a)-\lm^{k+1},
\text{  for all  } i=\overline{1,N},$$ which means that $h \in
int(H_{n+1}).$
\end{proof}

\begin{proof}[Proof of Lemma~\ref{l:convergence}]
   The proof of \textbf{(ii)} is a trivial application of monotonicity lemma \ref{l:monotonicity} and Condition \textbf{(B)} of the operator $Z_k$. \textbf{(i)} follows from Remark \ref{r:step-size}. \textbf{(iii)} As it was shown in \cite{bertsekas:2001,
bertsekas:1998} the Markovian operator $T_{\lm^k}$ is a contraction mapping with respect to a weighted $\max$-norm. Then it is easy to see (similar to the Remark \ref{r:step-size}) that the sequence $h^k$ and a sequence of fixed point of $\lm^k$-SSP $h_{\lm^k}$ both converge to $h^*$ with respect to this norm. From this and inequality \textbf{(ii)} we immediately obtain \textbf{(iii)}.
\end{proof}

\begin{proof}[Proof of Theorem~\ref{l:Pu_satisfy12}]
Condition (A) is satisfied trivially since $h +\al e\in
H_{\lm^{k}}$ for any $h\in H_{\lm^k}$ by the definition of $Z_k$
given in (\ref{eq:1DLP}). Now we have to show that $Z_k$ satisfies
condition (B). We know $\al=0$ is feasible to the linear program
(\ref{eq:1DLP}) since $h\in H_{\lm^k}$ (or $T_{\lm^k}h\ge h$). By
lemma \ref{lem:invariance} $T_{\lm^k}h \in H_{\lm^{k+1}}$, and we have $\al^* \ge 1$. Therefore, $Z_k h = h
+\al^* e  \ge h$.
\end{proof}

\begin{proof}[Proof of Theorem~\ref{thm:L_w_satisfy12}]
For $h\in H_{\lm^k}$, $Z_k h=h+\al^*(T_{\lm^k}h-h)$, where $\al^*$
is an optimal solution to the linear program in (\ref{eq:1DLP2}).
Since $h+\al^*(T_{\lm^k}h-h)$ is feasible to the linear program,
we have $T_{\lm^k}(h+\al^*(T_{\lm^k}h-h))\ge
h+\al^*(T_{\lm^k}h-h)$. Together with lemma ~\ref{lem:invariance}
this suffices Condition (A). By $T_{\lm^k}h \in H_{\lm^k}$, $\al=1$
is feasible. Since $T_{\lm^k}h\ge h$ and $\al=1$ is feasible,
$\al^* \ge 0$. Hence, Condition (B) is satisfied.
\end{proof}

\begin{proof}[Proof of Lemma~\ref{lem:invariance of splittings}]
The proof is similar to proofs of Lemma~\ref{lem:invariance}.
\end{proof}

\begin{proof} [Proof of Lemma~\ref{l:relationships_between_V's}]
In order to prove inclusion $H_{\lm^k} \subset H_{GS_{\lm^k}}$, it
is sufficient to show that if $h\le T_{\lm^k}h$, then $h\le
T_{GS_{\lm^k}}h$. For $h\in V$, let $g=T_{GS_{\lm^k}}{h}$ and
$f=T_{\lm^k}h$. Then, $f(j)\ge{}h(j)$ for all $j$ and $g(1)=f(1)$.
Assume that $g(k)\ge{}f(k)$ for all $k<i$, then
\begin{align*}
 g(i)&=\min_{a\in A(i)}\biggl(r(i,a)-\lm^k+\sum_{j<i}p_{ij}(a)g(j) +
      \sum_{j \ge i}p_{ij}(a)h(j) \biggr)\\
    &\ge\min_{a\in{}A(i)}\biggl(r(i,a)-\lm^k+\sum_{j<i}p_{ij}(a)f(j) +
      \sum_{j \ge i}p_{ij}(a)h(j) \biggr)\\
    &\ge\min_{a\in{}A(i)}\biggl(r(i,a)-\lm^k+\sum_{j<i}p_{ij}(a)h(j) +
      \sum_{j \ge i}p_{ij}(a)h(j) \biggr)=(T_{\lm^k}v)(i)=f(i).
\end{align*}
By induction, $f\le g$, implying $h\le T_{\lm^k}h=f\le g =
T_{GS_{\lm^k}}h$.
\end{proof}

\begin{proof}[Proof of Lemma~\ref{l:TX(VX)subsetV}]
For $h\in H_{GS_{\lm^k}}$, let $g= T_{GS_{\lm^k}}{h}$.  By
Lemma~\ref{l:monotonicity}, $h\le T_{GS_{\lm^k}}{}h=g$. By
replacing ``$\min_{a\in A(i)}$'' with inequalities, similar to the
argument used in the proof of
Lemma~\ref{l:relationships_between_V's}, and by $h\le g$, we have
\begin{align*}
 g(i) &\le{}r(i,a)-\lm^k+\sum_{j=1}^{i-1}p_{ij}(a)g(j)+\sum_{j=i}^{N-1}p_{ij}(a)h(j)
   \text{ for all $i\in{}\overline{2,N}$ for all $a\in{}A(i)$} \\
     &\le r(i,a)-\lm^k+\sum_{j=1}^{i-1}p_{ij}(a)g(j)+
  \sum_{j=i}^{N-1}p_{ij}(a)g(j) \text{ for all $i\in{}\overline{2,N}$ for all $a\in{}A(i)$,}
\end{align*}
which is equivalent to $g\le{}T_{\lm^k}g$, or $g\in{H_{\lm^k}}$.
\end{proof}

\begin{proof}[Proof of Theorem~\ref{thm:H_invariance_under_TGS}]
By Lemma~\ref{l:relationships_between_V's} and
Lemma~\ref{l:TX(VX)subsetV},
\[
T_{GS_{\lm^k}}H_{\lm^k} \subset T_{GS_{\lm^k}}H_{GS_{\lm^k}}
\subset H_{\lm^k}.
\]
\end{proof}


\begin{figure} [p]
  \begin{center}
  \setlength{\unitlength}{0.00053333in}
\begingroup\makeatletter\ifx\SetFigFont\undefined%
\gdef\SetFigFont#1#2#3#4#5{%
  \reset@font\fontsize{#1}{#2pt}%
  \fontfamily{#3}\fontseries{#4}\fontshape{#5}%
  \selectfont}%
\fi\endgroup%
{\renewcommand{\dashlinestretch}{30}
\begin{picture}(7524,6831)(0,-10)
\put(5037,2202){\blacken\ellipse{150}{150}}
\put(5037,2202){\ellipse{150}{150}}
\put(5262,1527){\blacken\ellipse{150}{150}}
\put(5262,1527){\ellipse{150}{150}}
\blacken\path(2817.000,6207.000)(2787.000,6327.000)(2757.000,6207.000)(2817.000,6207.000)
\path(2787,6327)(2787,702)
\thicklines
\path(1212,5877)(3087,5277)(3537,4977)
	(4512,3702)(5412,1077)
\thinlines
\dottedline{45}(4095,4887)(5820,12)
\dottedline{45}(2912,5859)(5312,2634)
\dottedline{45}(730,6804)(4180,4554)
\dottedline{45}(12,6253)(4137,4903)
\dottedline{45}(4098,4888)(5823,13)
\thicklines
\dashline{60.000}(5412,1077)(4062,5052)
\thinlines
\dottedline{45}(5262,927)(5262,5652)
\dottedline{45}(5037,1302)(5037,5802)
\thicklines
\path(5262,4377)(5262,4227)
\thinlines
\path(1212,4302)(7512,4302)
\blacken\path(7392.000,4272.000)(7512.000,4302.000)(7392.000,4332.000)(7392.000,4272.000)
\thicklines
\path(5037,4377)(5037,4227)
\path(4287,4377)(4287,4227)
\thinlines
\path(3076,4684)(3301,5059)
\blacken\path(3264.985,4940.666)(3301.000,5059.000)(3213.536,4971.536)(3264.985,4940.666)
\put(5487,1377){\makebox(0,0)[lb]{\smash{{{\SetFigFont{12}{14.4}{\rmdefault}{\mddefault}{\updefault}$h_{\lm^k}(n)$}}}}}
\put(5262,2127){\makebox(0,0)[lb]{\smash{{{\SetFigFont{12}{14.4}{\rmdefault}{\mddefault}{\updefault}$h_{\lm^{k+1}}(n)$}}}}}
\put(5337,4377){\makebox(0,0)[lb]{\smash{{{\SetFigFont{12}{14.4}{\rmdefault}{\mddefault}{\updefault}$\lm^k$}}}}}
\put(4812,4377){\makebox(0,0)[lb]{\smash{{{\SetFigFont{12}{14.4}{\rmdefault}{\mddefault}{\updefault}$\lm^{k+1}$}}}}}
\put(3762,4002){\makebox(0,0)[lb]{\smash{{{\SetFigFont{12}{14.4}{\rmdefault}{\mddefault}{\updefault}$\lm^*$}}}}}
\put(2862,4452){\makebox(0,0)[lb]{\smash{{{\SetFigFont{12}{14.4}{\rmdefault}{\mddefault}{\updefault}$h_{\lm}(n)$}}}}}
\put(4262,4377){\makebox(0,0)[lb]{\smash{{{\SetFigFont{12}{14.4}{\rmdefault}{\mddefault}{\updefault}$\lm^{k+2}$}}}}}
\end{picture}
}
  \end{center}
  \caption{$\lm$ update}
  \label{f:PAVI1}
\end{figure}

\bigskip

\begin{figure} [p]
  \begin{center}
  \setlength{\unitlength}{0.00053333in}
\begingroup\makeatletter\ifx\SetFigFont\undefined%
\gdef\SetFigFont#1#2#3#4#5{%
  \reset@font\fontsize{#1}{#2pt}%
  \fontfamily{#3}\fontseries{#4}\fontshape{#5}%
  \selectfont}%
\fi\endgroup%
{\renewcommand{\dashlinestretch}{30}
\begin{picture}(6924,6564)(0,-10)
\put(1212,237){\blacken\ellipse{150}{150}}
\put(1212,237){\ellipse{150}{150}}
\put(1287,912){\blacken\ellipse{150}{150}}
\put(1287,912){\ellipse{150}{150}}
\put(2787,2337){\blacken\ellipse{150}{150}}
\put(2787,2337){\ellipse{150}{150}}
\put(2787,2787){\blacken\ellipse{150}{150}}
\put(2787,2787){\ellipse{150}{150}}
\put(3387,3312){\blacken\ellipse{150}{150}}
\put(3387,3312){\ellipse{150}{150}}
\path(4287,1137)(4287,6537)
\blacken\path(4324.500,6417.000)(4287.000,6537.000)(4249.500,6417.000)(4324.500,6417.000)
\path(1212,4812)(6912,4812)
\blacken\path(6792.000,4774.500)(6912.000,4812.000)(6792.000,4849.500)(6792.000,4774.500)
\path(12,2562)(1662,3537)(3687,3837)(1512,12)
\thicklines
\dashline{60.000}(1212,162)(1287,912)(2787,2337)
	(2787,2787)(3387,3312)
\thinlines
\dottedline{60}(1638,1764)(4938,4689)
\dottedline{45}(344,52)(4019,3427)
\put(837,12){\makebox(0,0)[lb]{\smash{{{\SetFigFont{12}{14.4}{\rmdefault}{\mddefault}{\updefault}$\tilde{h}^{k}$}}}}}
\put(912,1062){\makebox(0,0)[lb]{\smash{{{\SetFigFont{12}{14.4}{\rmdefault}{\mddefault}{\updefault}$T\tilde{h}^{k}$}}}}}
\put(3012,2112){\makebox(0,0)[lb]{\smash{{{\SetFigFont{12}{14.4}{\rmdefault}{\mddefault}{\updefault}$\tilde{h}^{k+1}$}}}}}
\put(2487,2962){\makebox(0,0)[lb]{\smash{{{\SetFigFont{12}{14.4}{\rmdefault}{\mddefault}{\updefault}$T\tilde{h}^{k+1}$}}}}}
\put(387,2187){\makebox(0,0)[lb]{\smash{{{\SetFigFont{12}{14.4}{\rmdefault}{\mddefault}{\updefault}$H_{\lm^k}$}}}}}
\put(3612,3237){\makebox(0,0)[lb]{\smash{{{\SetFigFont{12}{14.4}{\rmdefault}{\mddefault}{\updefault}$\tilde{h}^{k+2}$}}}}}
\put(3762,3912){\makebox(0,0)[lb]{\smash{{{\SetFigFont{12}{14.4}{\rmdefault}{\mddefault}{\updefault}$h_{\lm^k}$}}}}}
\end{picture}
}
  \end{center}
  \caption{Projective AVI}
  \label{f:PAVI1}
\end{figure}

\bigskip

\begin{figure} [b]
\begin{center}
  \setlength{\unitlength}{0.00053333in}
\begingroup\makeatletter\ifx\SetFigFont\undefined%
\gdef\SetFigFont#1#2#3#4#5{%
  \reset@font\fontsize{#1}{#2pt}%
  \fontfamily{#3}\fontseries{#4}\fontshape{#5}%
  \selectfont}%
\fi\endgroup%
{\renewcommand{\dashlinestretch}{30}
\begin{picture}(6924,7014)(0,-10)
\put(1212,1437){\blacken\ellipse{150}{150}}
\put(1212,1437){\ellipse{150}{150}}
\put(1137,612){\blacken\ellipse{150}{150}}
\put(1137,612){\ellipse{150}{150}}
\put(1437,3837){\blacken\ellipse{150}{150}}
\put(1437,3837){\ellipse{150}{150}}
\put(2187,3837){\blacken\ellipse{150}{150}}
\put(2187,3837){\ellipse{150}{150}}
\put(3484,3890){\blacken\ellipse{150}{150}}
\put(3484,3890){\ellipse{150}{150}}
\path(4287,1587)(4287,6987)
\blacken\path(4324.500,6867.000)(4287.000,6987.000)(4249.500,6867.000)(4324.500,6867.000)
\path(1212,5262)(6912,5262)
\blacken\path(6792.000,5224.500)(6912.000,5262.000)(6792.000,5299.500)(6792.000,5224.500)
\path(12,3012)(1662,3987)(3687,4287)(1512,462)
\thicklines
\dashline{60.000}(1137,537)(1437,3837)(3462,3912)
\thinlines
\dottedline{45}(462,3762)(4062,3912)
\dottedline{60}(1062,12)(1512,4512)
\put(737,462){\makebox(0,0)[lb]{\smash{{{\SetFigFont{12}{14.4}{\rmdefault}{\mddefault}{\updefault}$\tilde{h}^{k}$}}}}}
\put(712,1512){\makebox(0,0)[lb]{\smash{{{\SetFigFont{12}{14.4}{\rmdefault}{\mddefault}{\updefault}$T\tilde{h}^{k}$}}}}}
\put(387,2637){\makebox(0,0)[lb]{\smash{{{\SetFigFont{12}{14.4}{\rmdefault}{\mddefault}{\updefault}$H_{\lm^k}$}}}}}
\put(3612,3687){\makebox(0,0)[lb]{\smash{{{\SetFigFont{12}{14.4}{\rmdefault}{\mddefault}{\updefault}$\tilde{h}^{k+2}$}}}}}
\put(3762,4362){\makebox(0,0)[lb]{\smash{{{\SetFigFont{12}{14.4}{\rmdefault}{\mddefault}{\updefault}$h_{\lm^k}$}}}}}
\put(1137,3987){\makebox(0,0)[lb]{\smash{{{\SetFigFont{12}{14.4}{\rmdefault}{\mddefault}{\updefault}$\tilde{h}^{k+1}$}}}}}
\put(2187,3462){\makebox(0,0)[lb]{\smash{{{\SetFigFont{12}{14.4}{\rmdefault}{\mddefault}{\updefault}$T\tilde{h}^{k+1}$}}}}}
\end{picture}
}  
\end{center}
\caption{Linear AVI} \label{f:LAVI}
\end{figure}

\begin{center}
\begin{table} [ht]
\fontsize{10}{8} \selectfont
\begin{tabular}{|c|c|c|c|c|c|c|c|} \hline
&{Density}$^\dag$&PAVI 1&PAVI 2 &PAVI 3&Bertsecas VI 1 $^\ddag$&Berstsecas VI 2$^\ddag$\\\hline
& 30 &69 &262 &\textbf{50} &546 & 1064 \\\cline{2-7}
& 40 & 99&204 &\textbf{54} & 421& 844 \\\cline{2-7}
& 50 & 81& 165&\textbf{35} & 321 & 675 \\\cline{2-7}
 \raisebox{1.5ex}[0cm][0cm]{Example 1} & 30& 62 &226 & \textbf{35}& 377& 728 \\\cline{2-7}
& 70& 108 &198 & \textbf{44}& 366& 720\\\cline{2-7}
& 80& 105 &260 & \textbf{42}& 422& 828 \\\cline{2-7}
& 90 & 102& 137& \textbf{32}& 327& 640 \\\hline
 & 60 &137 & 542& \textbf{44}& 760& 1525 \\\cline{2-7}
& 70 & 127& 604& \textbf{39}& 810& 1630\\\cline{2-7}
 \raisebox{1.9ex}[0cm][0cm]{Example 2} & 80& 128&448 & \textbf{40}&716 & 1504\\\cline{2-7}
& 90 &130 & 578&\textbf{40} &762 & 1546 \\\hline
 & 40 & 151& 358&\textbf{61} & 589 &  1145\\\cline{2-7}
& 50 &97 &354 & \textbf{47}&558 & 1139 \\\cline{2-7}
& 60 &116 &394 & \textbf{55} & 679& 1383 \\\cline{2-7}
 \raisebox{1.9ex}[0cm][0cm]{Example 3} & 70& 81 &354 & \textbf{38}& 555 & 1133 \\\cline{2-7}
& 80 &77 & 382& \textbf{32}& 553&  1111\\\cline{2-7}
& 90 & 82& 412& \textbf{36}& 551& 1153 \\\hline
 & 50 &118 &825 & \textbf{62}&1196 & 2730\\\cline{2-7}
& 60 &107 & 863& \textbf{48}& 1178 & 2591 \\\cline{2-7}
 \raisebox{0.0ex}[0cm][0cm]{Example 4} & 70& 119 & 990& \textbf{35}& 1358 & 3027\\\cline{2-7}
& 80 & 86& 947& \textbf{41}& 1319 & 2990 \\\cline{2-7}
& 90 & 97& 956& \textbf{39}& 1339 & 3058 \\\hline
\end{tabular}
\vspace{0.1in}

{\footnotesize $\dag$: The density of the transition probability matrix (\%).\\
$\ddag$: These two columns are the
number of iterations of the Bertsecas value iteration algorithm with improved bounds obtained from the phase 1 and without them.}

\vspace{-0.1in}
\caption{The number of iterations of the value iteration algorithms with and without an accelerating
operator applied to a family of MDPs from Examples~\ref{ex:dense1} -~\ref{ex:dense4}.} \label{tab:1}
\end{table}
\end{center}
\vspace{0.1in}

\begin{center}
\begin{table} [ht]
\fontsize{10}{8} \selectfont
\begin{tabular}{|c|c|c|c|c|c|c|c|} \hline
&{Density}$^\dag$&PAGS 1&PAGS 2 &PAGS 3&Bertsecas GS 1 $^\ddag$&Berstsecas GS 2$^\ddag$\\\hline
& 30 &69 & \textbf{62}& 63& 346 & 462 \\\cline{2-7}
& 40 & 99& \textbf{63}&71 & 283& 408\\\cline{2-7}
& 50 & 81& 63& \textbf{51}&  221& 333\\\cline{2-7}
 \raisebox{1.5ex}[0cm][0cm]{Example 1} & 30&62  & 51& \textbf{39}&240 & 391 \\\cline{2-7}
& 70& 108 &\textbf{64} &77 &230 & 344\\\cline{2-7}
& 80& 105 & \textbf{65}& 97& 253& 283 \\\cline{2-7}
& 90 & 102& \textbf{43}& 54& 198 & 302 \\\hline
 & 60 & 137&\textbf{43} &116 &383 & 570 \\\cline{2-7}
& 70 & 127&132 &\textbf{85} &402 & 605\\\cline{2-7}
 \raisebox{1.9ex}[0cm][0cm]{Example 2} & 80&128 &\textbf{87 }&106 &383 & 559\\\cline{2-7}
& 90 &130 & \textbf{62}& 103 & 383& 568 \\\hline
 & 40 &151 &\textbf{70} & 86 & 345 & 487 \\\cline{2-7}
& 50 & 97& \textbf{67}& 72 & 328 &  477\\\cline{2-7}
& 60 & \textbf{116}& 137& 125& 369 &  581\\\cline{2-7}
 \raisebox{1.9ex}[0cm][0cm]{Example 3} & 70& 81 &\textbf{41 }& 85&  324& 491 \\\cline{2-7}
& 80 &77 & \textbf{38}& 82& 302& 454 \\\cline{2-7}
& 90 &\textbf{82} & 84& 131& 314 & 470 \\\hline
 & 50 &118 & \textbf{46}&111 & 552& 787\\\cline{2-7}
& 60 &107 &140 & \textbf{97}&  545& 810 \\\cline{2-7}
 \raisebox{0.0ex}[0cm][0cm]{Example 4} & 70& 119 & \textbf{67}&94 & 584 & 807\\\cline{2-7}
& 80 &\textbf{86} & 114& 129&  572& 789 \\\cline{2-7}
& 90 &\textbf{97} & 98 &134 &  585&  837\\\hline
\end{tabular}
\vspace{0.1in}

{\footnotesize $\dag$: The density of the transition probability matrix (\%).\\
$\ddag$: These two columns are the
number of iterations of the Gauss-Seidel variant of Bertsecas algorithm  with  improved bounds obtained from the phase 1 and without them.}

\vspace{-0.1in}
\caption{The number of iterations of the value iteration algorithms with and without an accelerating
operator applied to a family of MDPs from Examples~\ref{ex:dense1} -~\ref{ex:dense4}.} \label{tab:2}
\end{table}
\end{center}
\vspace{0.1in}

\begin{center}
\begin{table} [ht]
\fontsize{9}{7} \selectfont
\begin{tabular}{|c|c|c|c|c|c|c|c|c|c|} \hline
 \multicolumn{2}{|c|}{$\al$}&\multicolumn{2}{|c|}{0.8}&\multicolumn{2}{|c|}{0.9}&\multicolumn{2}{|c|}{0.99}&\\\cline{1-8}
&{Dens.}$^\dag$&Min$^\ddag$&Max$^\ddag$&Min$^\ddag$&Max$^\ddag$&Min$^\ddag$&Max$^\ddag$&\raisebox{1.0ex}[0cm][0cm]{$\lm^*$}\\\hline
 & 30 &3.665  & 6.420& 4.053 & 5.437 & 4.412  & 4.551&4.443\\\cline{2-9}
& 40 & 3.550  & 7.020& 3.918  & 5.655& 4.262& 4.436&4.292\\\cline{2-9}
& 50 & 4.436  & 6.754& 4.300  & 5.718& 4.652  & 4.794&4.683\\\cline{2-9}
 \raisebox{1.9ex}[0cm][0cm]{Example 1} & 60 &  3.651  & 6.059& 4.009& 5.215& 4.342  & 4.463& 4.369\\\cline{2-9}
& 70 & 4.119  & 9.428& 4.504  & 7.170& 4.861& 5.129& 4.890\\\cline{2-9}
& 80 & 4.132    & 10.351& 4.646  & 7.784& 5.143  & 5.460&5.188 \\\cline{2-9}
& 90 & 3.238  & 7.954& 3.543& 5.901& 3.828  & 4.064&3.854 \\\hline
 & 60 & 8.296 & 17.677&  9.32 & 14.031& 10.266& 10.739&10.389 \\\cline{2-9}
& 70 & 8.469 & 17.98& 9.544& 14.317& 10.544  & 11.023& 10.674\\\cline{2-9}
 \raisebox{1.9ex}[0cm][0cm]{Example 2} &80&6.742  &15.461 &7.479 & 11.84& 8.149 & 8.585& 8.24\\\cline{2-9}
& 90 & 8.497& 18.748& 9.464& 14.613& 10.351  & 10.868&10.469 \\\hline
 & 40 & 5.349  & 16.011&5.978  & 11.316& 6.591  & 7.125&6.653 \\\cline{2-9}
& 50 & 4.640  & 9.781& 5.166  & 7.745& 5.654  & 5.912& 5.698\\\cline{2-9}
& 60 & 5.008  & 14.888& 5.601& 10.551& 6.157& 6.653& 6.207\\\cline{2-9}
 \raisebox{1.9ex}[0cm][0cm]{Example 3} & 70&  4.541& 9.931& 5.022  & 7.727& 5.471& 5.742
&5.509 \\\cline{2-9}
& 80 & 5.870    & 11.895 & 6.523  & 9.554& 7.126  & 7.430&7.187 \\\cline{2-9}
& 90 & 6.140  & 14.117& 6.821& 10.820& 7.446  & 7.846
& 7.505\\\hline
 & 50 & 5.879&14.868& 6.460    & 11.228& 7.097  & 7.575& 7.160\\\cline{2-9}
& 60 & 6.156  & 12.343& 6.587  & 11.086& 7.244  & 7.695& 7.311\\\cline{2-9}
 \raisebox{1.9ex}[0cm][0cm]{Example 4} & 70&  6.156  & 12.343& 6.904  & 10.008& 7.600  & 7.911& 7.673\\\cline{2-9}
& 80 & 5.873   &13.894& 6.551  & 10.574& 7.178  & 7.581&7.238 \\\cline{2-9}
& 90 & 6.583  & 17.910& 7.299  & 12.973& 7.950  & 8.518& 8.016\\\hline
 \end{tabular}

\vspace{0.1in} {\footnotesize $\dag$: The density of the transition probability matrix (\%).\\
$\ddag$: Two values in these columns correspond to the
upper and lower bounds of the optimal average reward $\lm^*$}
\vspace{-0.1in}
\caption{Comprising of the upper and lower bounds of the optimal average reward $\lm^*$ of family of MDPs from Examples~\ref{ex:dense1} -~\ref{ex:dense4} obtained by solving corresponding discouned MDP problems with various values of the discount factor $\al$.} \label{tab:3}
\end{table}
\end{center}


\begin{thebibliography}{OW}
\bibitem{bertsekas:2001} D.P. Bertsekas.
Dynamic Programming and Optimal Control. Belmont: Athena
Scientific, 2001.
\bibitem{bertsekas:1998} D. P. Bertsekas, A New Value Iteration Method
for the Average Cost Dynamic Programming Problem, {\sl SIAM J. on
Control and Optimization.} {\bf {Vol. 36}}, (1998), pp. 742-759.
\bibitem{DeFarias/VanRoy:OR03p850}
D.~D. Farias and B.~V. Roy.
\newblock The linear programming approach to approximate dynamic programming.
\newblock {\em Operations Research}, 51(6):850--856, 2003.
\bibitem{Derman:MDP70}
C.~Derman.
\newblock {\em Finite State Markovian Decision Processes}.
\newblock Academic Press, New York, 1970.
\bibitem{Lewis:2001}M. E. Lewis and M. L. Puterman.
\newblock A Probabilistic Analysis of Bias Optimality in Unichain Markov Decision Processes.
\newblock {\em IEEE Transactions on Automatic Control}, Vol. 46, Issue 1 (January), 96-100. 2001.
\bibitem{Puterman:MDP94} M.L. Puterman.
Markov Decision Processes: Discret Stochastic Dynamic Programming.
New York: Wiley, 1994.
\bibitem{Schweitzer/Seidmann:JMAA85p568}
P.~Schweitzer and A.~Seidmann.
\newblock Generalized polinomial approximation in markovian decision processes.
\newblock {\em Journal of Mathematical Analysis and Applications},
  110:568--582, 1985.
\bibitem{shlakhter:2005}O. Shlakhter, C.-G. Lee, D. Khmelev, and N. Jaber,
Acceleration Operators in the Value Iteration Algorithms for
Markov Decision Processes, subm. to publication. Preprint version can be obtained from http://front.math.ucdavis.edu/0506.5489
(2005).

\end{thebibliography}
\end{document}